\documentclass{amsart}
\usepackage{amsmath,amssymb}
\usepackage[dvips]{graphics}
\usepackage[all]{xy}
\newtheorem{Theorem}{Theorem}[section]
\newtheorem{Lemma}[Theorem]{Lemma}
\newtheorem{Proposition}[Theorem]{Proposition}
\newtheorem{Corollary}[Theorem]{Corollary}

\theoremstyle{definition}

\theoremstyle{remark}
\newtheorem{Remark}[Theorem]{Remark}

\def\labelenumi{\rm ({\makebox[0.8em][c]{\theenumi}})}
\def\theenumi{\arabic{enumi}}
\def\labelenumii{\rm {\makebox[0.8em][c]{(\theenumii)}}}
\def\theenumii{\roman{enumii}}

\renewcommand{\thefigure}
{\arabic{section}.\arabic{figure}}
\makeatletter
\@addtoreset{figure}{section}
\def\@thmcountersep{-}
\makeatother

\numberwithin{equation}{section}

\begin{document} 

\title[Conway-Gordon type theorem for $K_{3,3,1,1}$]{Conway-Gordon type theorem for the complete four-partite graph $K_{3,3,1,1}$}

\author{Hiroka Hashimoto}
\address{Division of Mathematics, Graduate School of Science, Tokyo Woman's Christian University, 2-6-1 Zempukuji, Suginami-ku, Tokyo 167-8585, Japan}
\email{etiscatbird@yahoo.co.jp}

\author{Ryo Nikkuni}
\address{Department of Mathematics, School of Arts and Sciences, Tokyo Woman's Christian University, 2-6-1 Zempukuji, Suginami-ku, Tokyo 167-8585, Japan}
\email{nick@lab.twcu.ac.jp}
\thanks{The second author was partially supported by Grant-in-Aid for Young Scientists (B) (No. 21740046), Japan Society for the Promotion of Science.}


\date{}



\begin{abstract}
We give a Conway-Gordon type formula for invariants of knots and links in a spatial complete four-partite graph $K_{3,3,1,1}$ in terms of the square of the linking number and the second coefficient of the Conway polynomial. As an application, we show that every rectilinear spatial $K_{3,3,1,1}$ contains a nontrivial Hamiltonian knot. 
\end{abstract}

\maketitle

\section{Introduction} 

Throughout this paper we work in the piecewise linear category. Let $G$ be a finite graph. An embedding $f$ of $G$ into the Euclidean $3$-space ${\mathbb R}^{3}$ is called a {\it spatial embedding} of $G$ and $f(G)$ is called a {\it spatial graph}. We denote the set of all spatial embeddings of $G$ by ${\rm SE}(G)$. We call a subgraph $\gamma$ of $G$ which is homeomorphic to the circle a {\it cycle} of $G$ and denote the set of all cycles of $G$ by $\Gamma(G)$. We also call a cycle of $G$ a {\it $k$-cycle} if it contains exactly $k$ edges and denote the set of all $k$-cycles of $G$ by $\Gamma_{k}(G)$. In particular, a $k$-cycle is said to be {\it Hamiltonian} if $k$ equals the number of all vertices of $G$. For a positive integer $n$, $\Gamma^{(n)}(G)$ denotes the set of all cycles of $G$ ($=\Gamma(G)$) if $n=1$ and the set of all unions of $n$ mutually disjoint cycles of $G$ if $n\ge 2$. For an element $\gamma$ in $\Gamma^{(n)}(G)$ and an element $f$ in ${\rm SE}(G)$, $f(\gamma)$ is none other than a knot in $f(G)$ if $n=1$ and an $n$-component link in $f(G)$ if $n\ge 2$. In particular, we call $f(\gamma)$ a {\it Hamiltonian knot} in $f(G)$ if $\gamma$ is a Hamiltonian cycle. 

For an edge $e$ of a graph $G$, we denote the subgraph $G\setminus {\rm int}e$ by $G-e$. Let $e=\overline{uv}$ be an edge of $G$ which is not a loop, where $u$ and $v$ are distinct end vertices of $e$. Then we call the graph which is obtained from $G-e$ by identifying $u$ and $v$ the {\it edge contraction of $G$ along $e$} and denote it by $G/e$. A graph $H$ is called a {\it minor} of a graph $G$ if there exists a subgraph $G'$ of $G$ and the edges $e_{1},e_{2},\ldots,e_{m}$ of $G'$ each of which is not a loop such that $H$ is obtained from $G'$ by a sequence of edge contractions along $e_{1},e_{2},\ldots,e_{m}$. A minor $H$ of $G$ is called a {\it proper minor} if $H$ does not equal $G$. Let ${\mathcal P}$ be a property of graphs which is {\it closed} under minor reductions; that is, for any graph $G$ which does not have ${\mathcal P}$, all minors of $G$ also do not have ${\mathcal P}$. A graph $G$ is said to be {\it minor-minimal} with respect to ${\mathcal P}$ if $G$ has ${\mathcal P}$ but all proper minors of $G$ do not have ${\mathcal P}$. Then it is known that there exist finitely many minor-minimal graphs with respect to ${\mathcal P}$ \cite{RS04}.

Let $K_{m}$ be the {\it complete graph} on $m$ vertices, namely the simple graph consisting of $m$ vertices in which every pair of distinct vertices is connected by exactly one edge. Then the following are very famous in spatial graph theory, which are called the Conway-Gordon theorems. 

\begin{Theorem}\label{CG2} {\rm (Conway-Gordon \cite{CG83})}
\begin{enumerate}
\item For any element $f$ in ${\rm SE}(K_{6})$, 
\begin{eqnarray}\label{CG_f1}
\sum_{\gamma\in \Gamma^{(2)}(K_{6})}{\rm lk}(f(\gamma)) \equiv 1 \pmod{2}, 
\end{eqnarray}
where ${\rm lk}$ denotes the {\it linking number}.

\item For any element $f$ in ${\rm SE}(K_{7})$, 
\begin{eqnarray}\label{CG_f2}
\sum_{\gamma\in \Gamma_{7}(K_{7})}a_{2}(f(\gamma)) \equiv 1 \pmod{2}, 
\end{eqnarray}
where $a_{2}$ denotes the second coefficient of the {\it Conway polynomial}.

\end{enumerate}
\end{Theorem}

A graph is said to be {\it intrinsically linked} if for any element $f$ in ${\rm SE}(G)$, there exists an element $\gamma$ in $\Gamma^{(2)}(G)$ such that $f(\gamma)$ is a nonsplittable $2$-component link, and to be {\it intrinsically knotted} if for any element $f$ in ${\rm SE}(G)$, there exists an element $\gamma$ in $\Gamma(G)$ such that $f(\gamma)$ is a nontrivial knot. Theorem \ref{CG2} implies that $K_{6}$ (resp. $K_{7}$) is intrinsically linked (resp. knotted). Moreover, the intrinsic linkedness (resp. knottedness) is closed under minor reductions \cite{NT85} (resp. \cite{FL88}), and $K_{6}$ (resp. $K_{7}$) is minor-minimal with respect to the intrinsically linkedness \cite{S84} (resp. knottedness \cite{MRS88}). 

A {\it $\triangle Y$-exchange} is an operation to obtain a new graph $G_{Y}$ from a graph $G_{\triangle}$ by removing all edges of a $3$-cycle $\triangle$ of $G_{\triangle}$ with the edges $\overline{uv},\overline{vw}$ and $\overline{wu}$, and adding a new vertex $x$ and connecting it to each of the vertices $u,v$ and $w$ as illustrated in Fig. \ref{Delta-Y} (we often denote $\overline{ux}\cup \overline{vx}\cup \overline{wx}$ by $Y$). A {\it $Y \triangle$-exchange} is the reverse of this operation. We call the set of all graphs obtained from a graph $G$ by a finite sequence of $\triangle Y$ and $Y \triangle$-exchanges the {\it $G$-family} and denote it by ${\mathcal F}(G)$. In particular, we denote the set of all graphs obtained from $G$ by a finite sequence of $\triangle Y$-exchanges by ${\mathcal F}_{\triangle}(G)$. For example, it is well known that the $K_{6}$-family consists of exactly seven graphs as illustrated in Fig. \ref{Petersen}, where an arrow between two graphs indicates the application of a single $\triangle Y$-exchange. Note that ${\mathcal F}_{\triangle}(K_{6})={\mathcal F}(K_{6}) \setminus \left\{P_{7}\right\}$. Since $P_{10}$ is isomorphic to the {\it Petersen graph}, the $K_{6}$-family is also called the {\it Petersen family}. It is also well known that the $K_{7}$-family consists of exactly twenty graphs, and there exist exactly six graphs in the $K_{7}$-family each of which does not belong to ${\mathcal F}_{\triangle}(K_{7})$. Then the intrinsic linkedness and the intrinsic knottedness behave well under $\triangle Y$-exchanges as follows. 

\begin{Proposition}\label{MRS} 
{\rm (Sachs \cite{S84})} 
\begin{enumerate}
\item If $G_{\triangle}$ is intrinsically linked, then $G_{Y}$ is also intrinsically linked. 
\item If $G_{\triangle}$ is intrinsically knotted, then $G_{Y}$ is also intrinsically knotted. 
\end{enumerate}
\end{Proposition}

\begin{figure}[htbp]
      \begin{center}
\scalebox{0.45}{\includegraphics*{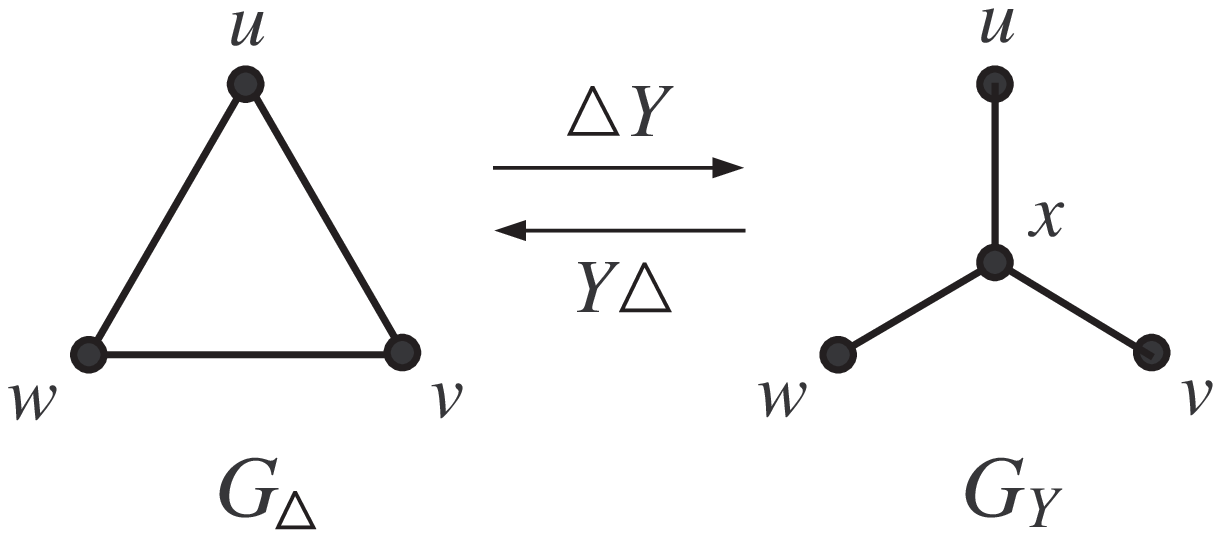}}
      \end{center}
   \caption{}
  \label{Delta-Y}
\end{figure} 
\begin{figure}[htbp]
      \begin{center}
\scalebox{0.41}{\includegraphics*{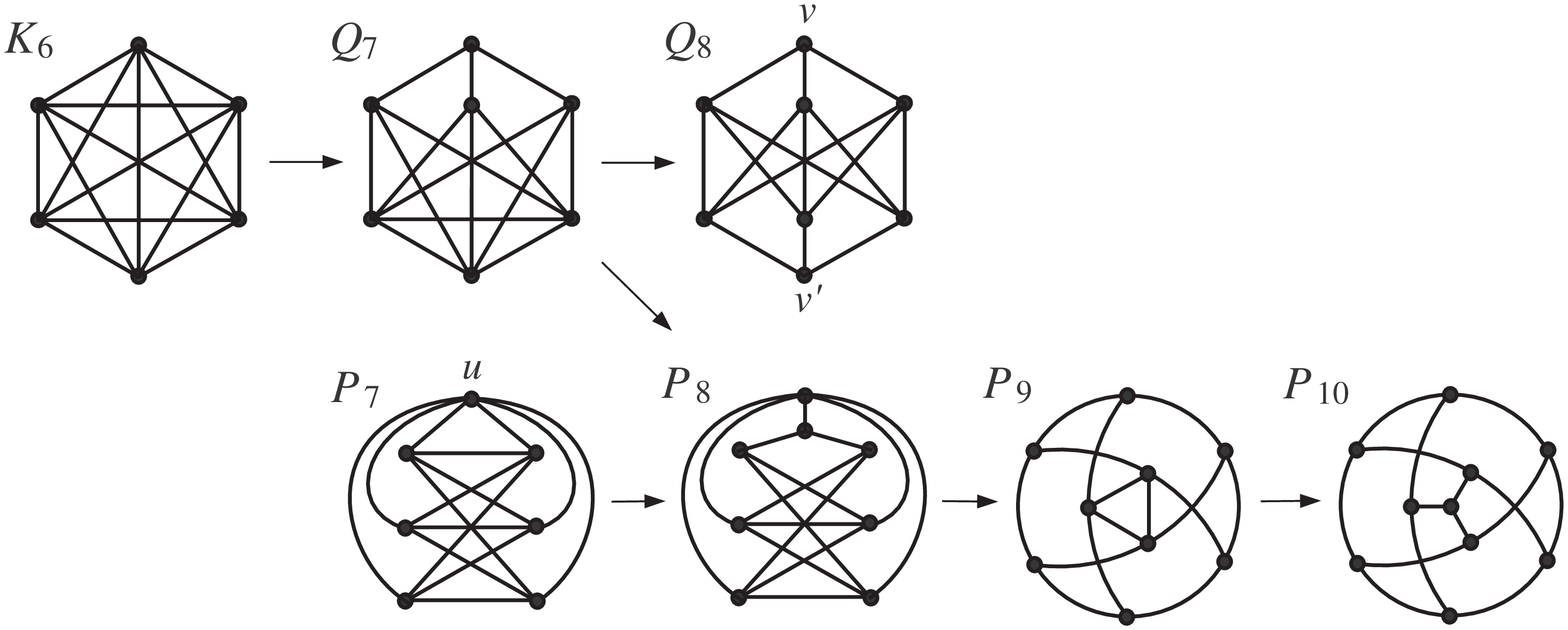}}
      \end{center}
   \caption{}
  \label{Petersen}
\end{figure} 

Proposition \ref{MRS} implies that any element in ${\mathcal F}_{\triangle}(K_{6})$ (resp. ${\mathcal F}_{\triangle}(K_{7})$) is intrinsically linked (resp. knotted). In particular, Robertson-Seymour-Thomas showed that the set of all minor-minimal intrinsically linked graphs equals the $K_{6}$-family, so the converse of Proposition \ref{MRS} (1) is also true \cite{RST95}. On the other hand, it is known that any element in ${\mathcal F}_{\triangle}(K_{7})$ is minor-minimal with respect to the intrinsic knottedness \cite{KS92}, but any element in ${\mathcal F}(K_{7}) \setminus {\mathcal F}_{\triangle}(K_{7})$ is not intrinsically knotted \cite{FN08}, \cite{HNTY10}, \cite{GMN}, so the converse of Proposition \ref{MRS} (2) is not true. Moreover, there exists a minor-minimal intrinsically knotted graph which does not belong to ${\mathcal F}_{\triangle}(K_{7})$ as follows. Let $K_{n_{1},n_{2},\ldots,n_{m}}$ be the {\it complete $m$-partite graph}, namely the simple graph whose vertex set can be decomposed into $m$ mutually disjoint nonempty sets $V_{1},V_{2},\ldots,V_{m}$ where the number of elements in $V_{i}$ equals $n_{i}$ such that no two vertices in $V_{i}$ are connected by an edge and every pair of vertices in the distinct sets $V_{i}$ and $V_{j}$ is connected by exactly one edge, see Fig. \ref{K331K3311} which illustrates $K_{3,3}$, $K_{3,3,1}$ and $K_{3,3,1,1}$. Note that $K_{3,3,1}$ is isomorphic to $P_{7}$ in the $K_{6}$-family, namely $K_{3,3,1}$ is a minor-minimal intrinsically linked graph. On the other hand, Motwani-Raghunathan-Saran claimed in \cite{MRS88} that it may be proven that $K_{3,3,1,1}$ is intrinsically knotted by using the same technique of Theorem \ref{CG2}, namely, by showing that for any element in ${\rm SE}(K_{3,3,1,1})$, the sum of $a_{2}$ over all of the Hamiltonian knots is always congruent to one modulo two. But Kohara-Suzuki showed in \cite{KS92} that the claim did not hold; that is, the sum of $a_{2}$ over all of the Hamiltonian knots is dependent to each element in ${\rm SE}(K_{3,3,1,1})$. Actually, they demonstrated the specific two elements $f_{1}$ and $f_{2}$ in ${\rm SE}(K_{3,3,1,1})$ as illustrated in Fig. \ref{K3311emb}. Here $f_{1}(K_{3,3,1,1})$ contains exactly one nontrivial knot $f_{1}(\gamma_{0})$ ($=$ a trefoil knot, $a_{2}=1$) which is drawn by bold lines, where $\gamma_{0}$ is an element in $\Gamma_{8}(K_{3,3,1,1})$, and $f_{2}(K_{3,3,1,1})$ contains exactly two nontrivial knots $f_{2}(\gamma_{1})$ and $f_{2}(\gamma_{2})$ ($=$ two trefoil knots) which are drawn by bold lines, where $\gamma_{1}$ and $\gamma_{2}$ are elements in $\Gamma_{8}(K_{3,3,1,1})$. Thus the situation of the case of $K_{3,3,1,1}$ is different from the case of $K_{7}$. By using another technique different from Conway-Gordon's, Foisy proved the following. 

\begin{figure}[htbp]
      \begin{center}
\scalebox{0.45}{\includegraphics*{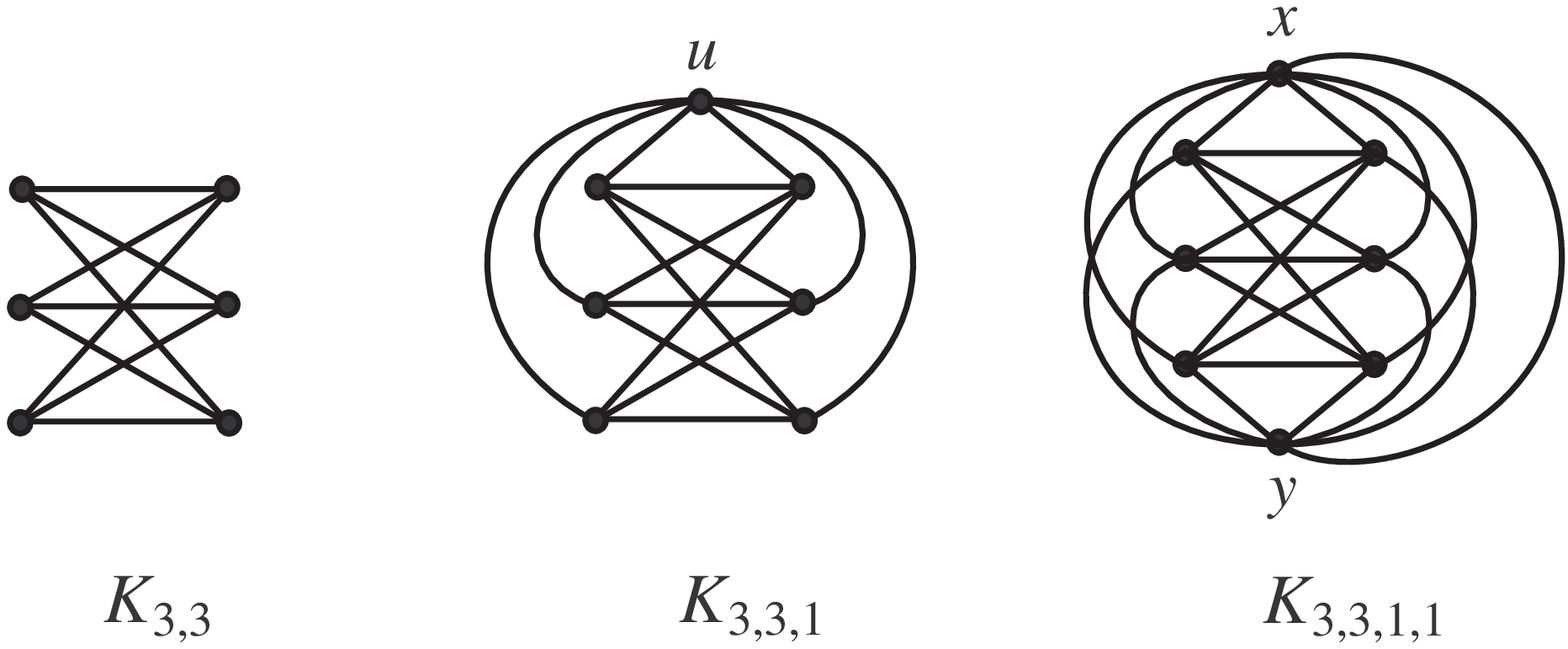}}
      \end{center}
   \caption{}
  \label{K331K3311}
\end{figure} 
\begin{figure}[htbp]
      \begin{center}
\scalebox{0.45}{\includegraphics*{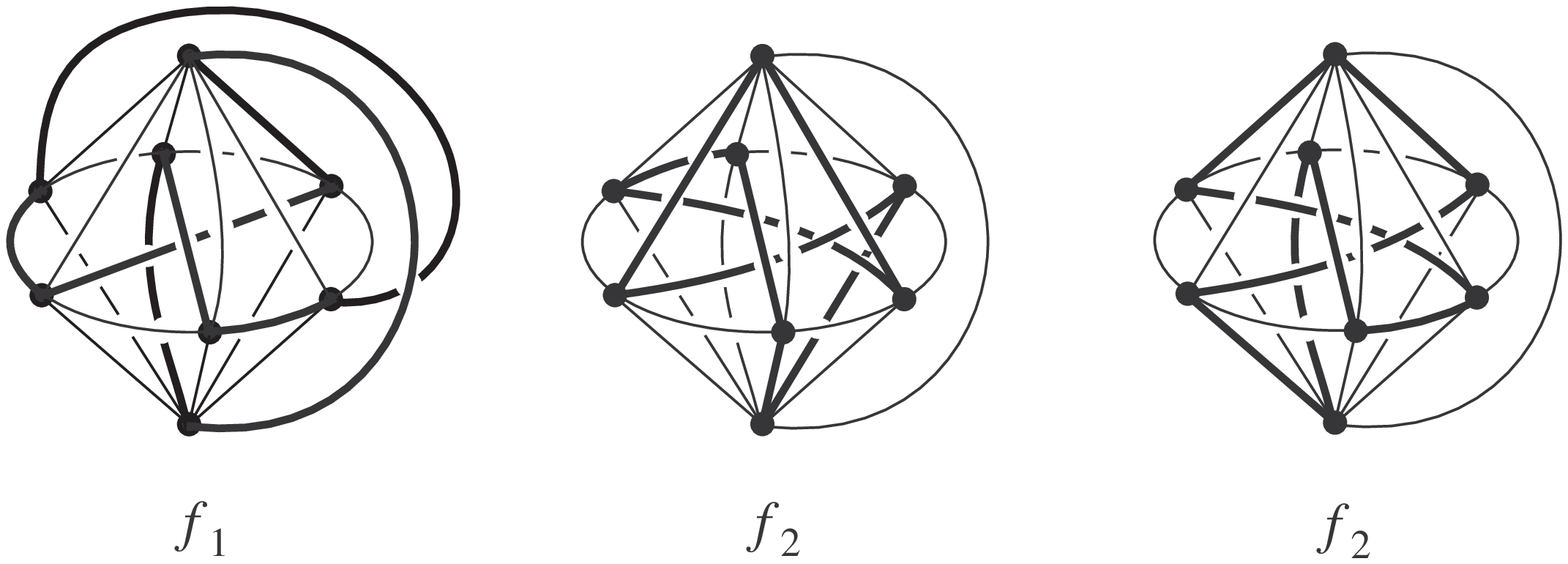}}
      \end{center}
   \caption{}
  \label{K3311emb}
\end{figure} 

\begin{Theorem}\label{ik_k3311}
{\rm (Foisy \cite{F02})} For any element $f$ in ${\rm SE}(K_{3,3,1,1})$, there exists an element $\gamma$ in $\cup_{k=4}^{8}\Gamma_{k}(K_{3,3,1,1})$ such that $a_{2}(f(\gamma))\equiv 1\pmod{2}$. 
\end{Theorem}

Theorem \ref{ik_k3311} implies that $K_{3,3,1,1}$ is intrinsically knotted. Moreover, Proposition \ref{MRS} (2) and Theorem \ref{ik_k3311} implies that any element $G$ in ${\mathcal F}_{\triangle}(K_{3,3,1,1})$ is also intrinsically knotted. It is known that there exist exactly twenty six elements in ${\mathcal F}_{\triangle}(K_{3,3,1,1})$. Since Kohara-Suzuki pointed out that each of the proper minors of $G$ is not intrinsically knotted \cite{KS92}, it follows that any element in ${\mathcal F}_{\triangle}(K_{3,3,1,1})$ is minor-minimal with respect to the intrinsic knottedness. Note that a $\triangle Y$-exchange does not change the number of edges of a graph. Since $K_{7}$ and $K_{3,3,1,1}$ have different numbers of edges, the families ${\mathcal F}_{\triangle}(K_{7})$ and ${\mathcal F}_{\triangle}(K_{3,3,1,1})$ are disjoint. 

Our first purpose in this article is to refine Theorem \ref{ik_k3311} by giving a kind of Conway-Gordon type formula for $K_{3,3,1,1}$ not over ${\mathbb Z}_{2}$, but integers ${\mathbb Z}$. In the following, $\Gamma_{k,l}^{(2)}(G)$ denotes the set of all unions of two disjoint cycles of a graph $G$ consisting of a $k$-cycle and an $l$-cycle, and $x$ and $y$ denotes the two vertices of $K_{3,3,1,1}$ with valency seven. Then we have the following.

\begin{Theorem}\label{main_theorem_k3311}
\begin{enumerate}
\item For any element $f$ in ${\rm SE}(K_{3,3,1,1})$, 
\begin{eqnarray}\label{maineq}
&&4\sum_{\gamma\in \Gamma_{8}(K_{3,3,1,1})}a_{2}(f(\gamma))
-4\sum_{\substack{{\gamma\in \Gamma_{7}(K_{3,3,1,1})} \\ {\left\{x,y\right\}\not\subset \gamma}}}a_{2}(f(\gamma))\\
&&-4\sum_{\gamma\in \Gamma_{6}'}a_{2}(f(\gamma))
-4\sum_{\substack{{\gamma\in \Gamma_{5}(K_{3,3,1,1})} \\ {\left\{x,y\right\}\not\subset \gamma}}}a_{2}(f(\gamma))\nonumber \\
&=&
\sum_{\lambda\in \Gamma_{3,5}^{(2)}(K_{3,3,1,1})}{\rm lk}(f(\lambda))^{2}
+2\sum_{\lambda\in \Gamma_{4,4}^{(2)}(K_{3,3,1,1})}{\rm lk}(f(\lambda))^{2}
-18, \nonumber
\end{eqnarray}
where $\Gamma_{6}'$ is a specific proper subset of $\Gamma_{6}(K_{3,3,1,1})$ which does not depend on $f$, see (\ref{gamma6}). 

\item For any element $f$ in ${\rm SE}(K_{3,3,1,1})$, 
\begin{eqnarray}\label{main_ineq}
\sum_{\lambda\in \Gamma_{3,5}^{(2)}(K_{3,3,1,1})}{\rm lk}(f(\lambda))^{2}
+2\sum_{\lambda\in \Gamma_{4,4}^{(2)}(K_{3,3,1,1})}{\rm lk}(f(\lambda))^{2}
\ge 22. 
\end{eqnarray}
\end{enumerate}
\end{Theorem}

We prove Theorem \ref{main_theorem_k3311} in the next section. By combining Theorem \ref{main_theorem_k3311} (1) and (2), we immediately have the following. 

\begin{Corollary}\label{main_theorem_k3311_cor}
For any element $f$ in ${\rm SE}(K_{3,3,1,1})$, 
\begin{eqnarray}\label{k3311_ineq}
&&\sum_{\gamma\in \Gamma_{8}(K_{3,3,1,1})}a_{2}(f(\gamma))
-\sum_{\substack{{\gamma\in \Gamma_{7}(K_{3,3,1,1})} \\ {\left\{x,y\right\}\not\subset \gamma}}}a_{2}(f(\gamma)) \\
&&-\sum_{\gamma\in \Gamma_{6}'}a_{2}(f(\gamma))
-\sum_{\substack{{\gamma\in \Gamma_{5}(K_{3,3,1,1})} \\ {\left\{x,y\right\}\not\subset \gamma}}}a_{2}(f(\gamma))
\ge 1. \nonumber
\end{eqnarray}
\end{Corollary}

Corollary \ref{main_theorem_k3311_cor} gives an alternative proof of the fact that $K_{3,3,1,1}$ is intrinsically knotted. Moreover, Corollary \ref{main_theorem_k3311_cor} refines Theorem \ref{ik_k3311} by identifying the cycles that might be nontrivial knots in $f(K_{3,3,1,1})$. 

\begin{Remark}\label{bestp}
We see the left side of (\ref{k3311_ineq}) is not always congruent to one modulo two by considering two elements $f_{1}$ and $f_{2}$ in ${\rm SE}(K_{3,3,1,1})$ as illustrated in Fig. \ref{K3311emb}. Thus Corollary \ref{main_theorem_k3311_cor} shows that the argument over ${\mathbb Z}$ has a nice advantage. In particular, $f_{1}$ gives the best possibility for (\ref{k3311_ineq}), and therefore for (\ref{main_ineq}) by Theorem \ref{main_theorem_k3311} (1). Actually $f_{1}(K_{3,3,1,1})$ contains exactly fourteen nontrivial links all of which are Hopf links, where the six of them are the images of elements in $\Gamma_{3,5}^{(2)}(K_{3,3,1,1})$ by $f_{1}$ and the eight of them are the images of elements in $\Gamma_{4,4}^{(2)}(K_{3,3,1,1})$ by $f_{1}$. 
\end{Remark}

As we said before, any element $G$ in ${\mathcal F}_{\triangle}(K_{7})\cup {\mathcal F}_{\triangle}(K_{3,3,1,1})$ is a minor-minimal intrinsically knotted graph. If $G$ belongs to ${\mathcal F}_{\triangle}(K_{7})$, then it is known that Conway-Gordon type formula over ${\mathbb Z}_{2}$ as in Theorem \ref{CG2} also holds for $G$ as follows. 

\begin{Theorem}\label{TY_main_cor} 
{\rm (Nikkuni-Taniyama \cite{NT12})} 
Let $G$ be an element in ${\mathcal F}_{\triangle}(K_{7})$. Then, there exists a map $\omega$ from $\Gamma(G)$ to ${\mathbb Z}_{2}$ such that for any element $f$ in ${\rm SE}(G)$, 
\begin{eqnarray*}
\sum_{\gamma\in \Gamma(G)}\omega(\gamma)a_{2}(f(\gamma)) \equiv 1 \pmod{2}.
\end{eqnarray*}
\end{Theorem}

Namely, for any element $G$ in ${\mathcal F}_{\triangle}(K_{7})$, there exists a subset $\Gamma$ of $\Gamma(G)$ which depends on only $G$ such that for any element $f$ in ${\rm SE}(G)$, the sum of $a_{2}$ over all of the images of the elements in $\Gamma$ by $f$ is odd. On the other hand, if $G$ belongs to ${\mathcal F}_{\triangle}(K_{3,3,1,1})$, we have a Conway-Gordon type formula over ${\mathbb Z}$ for $G$ as in Corollary \ref{main_theorem_k3311_cor} as follows. We prove it in section $3$. 

\begin{Theorem}\label{NT_main_cor} 
Let $G$ be an element in ${\mathcal F}_{\triangle}(K_{3,3,1,1})$. Then, there exists a map $\omega$ from $\Gamma(G)$ to ${\mathbb Z}$ such that for any element $f$ in ${\rm SE}(G)$, 
\begin{eqnarray*}
\sum_{\gamma\in \Gamma(G)}\omega(\gamma)a_{2}(f(\gamma))
\ge 1.
\end{eqnarray*}
\end{Theorem}

Our second purpose in this article is to give an application of Theorem \ref{main_theorem_k3311} to the theory of rectilinear spatial graphs. A spatial embedding $f$ of a graph $G$ is said to be {\it rectilinear} if for any edge $e$ of $G$, $f(e)$ is a straight line segment in ${\mathbb R}^{3}$. We denote the set of all rectilinear spatial embeddings of $G$ by ${\rm RSE}(G)$. We can see that any simple graph has a rectilinear spatial embedding by taking all of the vertices on the spatial curve $(t,t^{2},t^{3})$ in ${\mathbb R}^{3}$ and connecting every pair of two adjacent vertices by a straight line segment. Rectilinear spatial graphs appear in polymer chemistry as a mathematical model for chemical compounds, see \cite{adams} for example. Then by an application of Theorem \ref{main_theorem_k3311}, we have the following. 

\begin{Theorem}\label{main_theorem_recti}
For any element $f$ in ${\rm RSE}(K_{3,3,1,1})$, 
\begin{eqnarray*}
\sum_{\gamma\in \Gamma_{8}(K_{3,3,1,1})}a_{2}(f(\gamma))
\ge 1. 
\end{eqnarray*}
\end{Theorem}

We prove Theorem \ref{main_theorem_recti} in section $4$. As a corollary of Theorem \ref{main_theorem_recti}, we immediately have the following. 

\begin{Corollary}\label{main_cor_recti}
For any element $f$ in ${\rm RSE}(K_{3,3,1,1})$, there exists a Hamiltonian cycle $\gamma$ of $K_{3,3,1,1}$ such that $f(\gamma)$ is a nontrivial knot with $a_{2}(f(\gamma))>0$. 
\end{Corollary}

Corollary \ref{main_cor_recti} is an affirmative answer to the question of Foisy-Ludwig \cite[\sc Question 5.8]{FL09} which asks whether the image of every rectilinear spatial embedding of $K_{3,3,1,1}$ always contains a nontrivial Hamiltonian knot. 

\begin{Remark}\label{conc_rem}
\begin{enumerate}
\item In \cite[\sc Question 5.8]{FL09}, Foisy-Ludwig also asked that whether the image of every spatial embedding of $K_{3,3,1,1}$ (which may not be rectilinear) always contains a nontrivial Hamiltonian knot. As far as the authors know, it is still open. 
\item In addition to the elements in ${\mathcal F}_{\triangle}(K_{7})\cup {\mathcal F}_{\triangle}(K_{3,3,1,1})$, many minor-minimal intrinsically knotted graph are known \cite{F04}, \cite{GMN}. In particular, it has been announced by Goldberg-Mattman-Naimi that all of the thirty two elements in ${\mathcal F}(K_{3,3,1,1})\setminus {\mathcal F}_{\triangle}(K_{3,3,1,1})$ are minor-minimal intrinsically knotted graphs \cite{GMN}. Note that their method is based on Foisy's idea in the proof of Theorem \ref{ik_k3311} with the help of a computer. 
\end{enumerate}
\end{Remark}

\section{Conway-Gordon type formula for $K_{3,3,1,1}$} 

To prove Theorem \ref{main_theorem_k3311}, we recall a Conway-Gordon type formula over ${\mathbb Z}$ for a graph in the $K_{6}$-family which is as below. 

\begin{Theorem}\label{petersen_refine}
Let $G$ be an element in ${\mathcal F}(K_{6})$. Then there exist a map $\omega$ from $\Gamma(G)$ to ${\mathbb Z}$ such that for any element $f$ in ${\rm SE}(G)$, 
\begin{eqnarray}\label{pet_ref_eq}
2\sum_{\gamma\in \Gamma(G)}\omega(\gamma)a_{2}(f(\gamma))
=
\sum_{\gamma\in \Gamma^{(2)}(G)}{\rm lk}(f(\gamma))^{2}
-1. 
\end{eqnarray}
\end{Theorem}

We remark here that Theorem \ref{petersen_refine} was shown by Nikkuni (for the case $G=K_{6}$) \cite{N09b}, O'Donnol ($G=P_{7}$) \cite{D10} and Nikkuni-Taniyama (for the others) \cite{NT12}. In particular, we use the following explicit formulae for $Q_{8}$ and $P_{7}$ in the proof of Theorem \ref{main_theorem_k3311}. For the other cases, see Hashimoto-Nikkuni \cite{HN12a}. 

\begin{Theorem}\label{petersen_refine_pq} 
\begin{enumerate}
\item {\rm (Hashimoto-Nikkuni \cite{HN12a})} 
For any element $f$ in ${\rm SE}(Q_{8})$, 
\begin{eqnarray*}
&&2\sum_{\gamma\in \Gamma_{7}(P_{7})}a_{2}(f(\gamma))
+2\sum_{\substack{{\gamma\in \Gamma_{6}(Q_{8})} \\ {v,v'\not\in\gamma}}}a_{2}(f(\gamma))
-2\sum_{\substack{{\gamma\in \Gamma_{6}(Q_{8})} \\ {\gamma\cap\{v,v'\}\neq \emptyset}}}a_{2}(f(\gamma))\\
&=&
\sum_{\gamma\in \Gamma_{4,4}^{(2)}(Q_{8})}{\rm lk}(f(\gamma))^{2}-1,
\end{eqnarray*}
where $v$ and $v'$ are exactly two vertices of $Q_{8}$ with valency three. 

\item {\rm (O'Donnol \cite{D10})} For any element $f$ in ${\rm SE}(P_{7})$, 
\begin{eqnarray*}
&&2\sum_{\gamma\in \Gamma_{7}(P_{7})}a_{2}(f(\gamma))
-4\sum_{\substack{{\gamma\in \Gamma_{6}(P_{7})} \\ {u\not\in\gamma}}}a_{2}(f(\gamma))
-2\sum_{\gamma\in \Gamma_{5}(P_{7})}a_{2}(f(\gamma))\\
&=&
\sum_{\gamma\in \Gamma_{3,4}^{(2)}(P_{7})}{\rm lk}(f(\gamma))^{2}-1,
\end{eqnarray*}
where $u$ is the vertex of $P_{7}$ with valency six. 
\end{enumerate}
\end{Theorem}

By taking the modulo two reduction of (\ref{pet_ref_eq}), we immediately have the following fact containing Theorem \ref{CG2} (1). 

\begin{Corollary}\label{CG_refine_Do} 
{\rm (Sachs \cite{S84}, Taniyama-Yasuhara \cite{TY01})} 
Let $G$ be an element in ${\mathcal F}(K_{6})$. Then, for any element $f$ in ${\rm SE}(G)$,  
\begin{eqnarray*}
\sum_{\gamma\in \Gamma^{(2)}(G)}{\rm lk}(f(\gamma))
\equiv 1 \pmod{2}. 
\end{eqnarray*}
\end{Corollary}

Now we give labels for the vertices of $K_{3,3,1,1}$ as illustrated in the left figure in Fig. \ref{GxGy}. We also call the vertices $1,3,5$ and $2,4,6$ the {\it black vertices} and the {\it white vertices}, respectively. We regard $K_{3,3}$ as the subgraph of $K_{3,3,1,1}$ induced by all of the white and black vertices. Let $G_{x}$ and $G_{y}$ be two subgraphs of $K_{3,3,1,1}$ as illustrated in Fig. \ref{GxGy} (1) and (2), respectively. Since each of $G_{x}$ and $G_{y}$ is isomorphic to $P_{7}$, by applying Theorem \ref{petersen_refine_pq} (2) to $f|_{G_{x}}$ and $f|_{G_{y}}$ for an element $f$ in ${\rm SE}(K_{3,3,1,1})$, it follows that   
\begin{eqnarray}\label{gx}
&&2\sum_{\gamma\in \Gamma_{7}(G_{x})}a_{2}(f(\gamma))
-4\sum_{\gamma\in \Gamma_{6}(K_{3,3})}a_{2}(f(\gamma))
-2\sum_{\gamma\in \Gamma_{5}(G_{x})}a_{2}(f(\gamma))\\
&=&
\sum_{\gamma\in \Gamma_{3,4}^{(2)}(G_{x})}{\rm lk}(f(\gamma))^{2}-1,\nonumber
\end{eqnarray}
\begin{eqnarray}\label{gy}
&&2\sum_{\gamma\in \Gamma_{7}(G_{y})}a_{2}(f(\gamma))
-4\sum_{\gamma\in \Gamma_{6}(K_{3,3})}a_{2}(f(\gamma))
-2\sum_{\gamma\in \Gamma_{5}(G_{y})}a_{2}(f(\gamma))\\
&=&
\sum_{\gamma\in \Gamma_{3,4}^{(2)}(G_{y})}{\rm lk}(f(\gamma))^{2}-1. \nonumber
\end{eqnarray}

\begin{figure}[htbp]
      \begin{center} 
\scalebox{0.5}{\includegraphics*{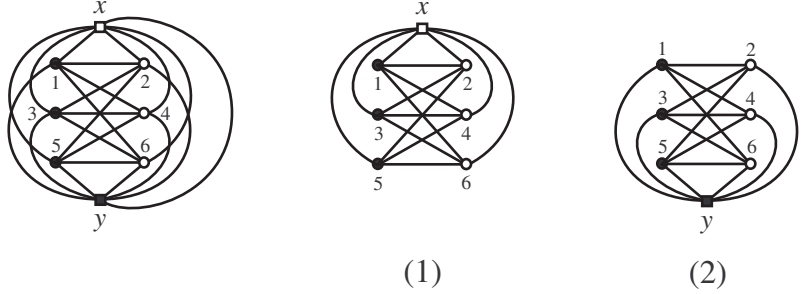}}
      \end{center}
   \caption{(1) $G_{x}$, (2) $G_{y}$}
  \label{GxGy}
\end{figure}

Let $\gamma$ be an element in $\Gamma(K_{3,3,1,1})$ which is a $8$-cycle or a $6$-cycle containing $x$ and $y$. Then we say that $\gamma$ is of {\it Type A} if the neighbor vertices of $x$ in $\gamma$ consist of both a black vertex and a white vertex (if and only if the neighbor vertices of $y$ in $\gamma$ consist of both a black vertex and a white vertex), of {\it Type B} if the neighbor vertices of $x$ in $\gamma$ consist of only black (resp. white) vertices and the neighbor vertices of $y$ in $\gamma$ consist of only white (resp. black) vertices, and of {\it Type C} if $\gamma$ contains the edge $\overline{xy}$. Moreover, we say that an element $\gamma$ in $\Gamma_{6}(K_{3,3,1,1})$ containing $x$ and $y$ is of {\it Type D} if the neighbor vertices of $x$ and $y$ in $\gamma$ consist of only black or only white vertices. Note that any element in $\Gamma_{8}(K_{3,3,1,1})$ is of Type A, B or C, and any element in $\Gamma_{6}(K_{3,3,1,1})$ containing $x$ and $y$ is of Type A, B, C or D. On the other hand, let $\lambda$ be an element in $\Gamma_{4,4}^{(2)}(K_{3,3,1,1})$. Then we say that $\lambda$ is of {\it Type A} if $\lambda$ does not contain the edge $\overline{xy}$ and both $x$ and $y$ are contained in either connected component of $\lambda$, of {\it Type B} if $x$ and $y$ are contained in different connected components of $\lambda$, and of {\it Type C} if $\lambda$ contains the edge $\overline{xy}$. Note that any element in $\Gamma_{4,4}^{(2)}(K_{3,3,1,1})$ is of Type A, B or C. Then the following three lemmas hold. 

\begin{Lemma}\label{l1}
For any element $f$ in ${\rm SE}(K_{3,3,1,1})$,  
\begin{eqnarray}\label{fxij13}
&&\sum_{\lambda\in \Gamma_{3,5}^{(2)}(K_{3,3,1,1})}{\rm lk}(f(\lambda))^{2}
+
2\sum_{\substack{
{\lambda\in \Gamma_{4,4}^{(2)}(K_{3,3,1,1})} 
\\ {\rm Type\ A}
}
}{\rm lk}(f(\lambda))^{2} \\
&=&
4\sum_{\substack{
{\gamma\in \Gamma_{8}(K_{3,3,1,1})} 
\\ {\rm Type\ A}
}
}a_{2}(f(\gamma))
-4\left\{
\sum_{\gamma\in \Gamma_{7}(G_{x})}a_{2}(f(\gamma))
+\sum_{\gamma\in \Gamma_{7}(G_{y})}a_{2}(f(\gamma))
\right\}\nonumber\\
&&+8\sum_{\gamma\in \Gamma_{6}(K_{3,3})}a_{2}(f(\gamma))
-4\sum_{\substack{
{\gamma\in \Gamma_{6}(K_{3,3,1,1})} 
\\ x,y\in \gamma,\ {\rm Type\ A}
}
}a_{2}(f(\gamma))\nonumber\\
&&-4\left\{
\sum_{\gamma\in \Gamma_{5}(G_{x})}a_{2}(f(\gamma))
+\sum_{\gamma\in \Gamma_{5}(G_{y})}a_{2}(f(\gamma))
\right\}
+10.  \nonumber
\end{eqnarray}
\end{Lemma}

\begin{proof}
For $i=1,3,5$ and $j=2,4,6$, let us consider subgraphs $F_{x}^{(ij)}=(G_{x}-\overline{ij})\cup \overline{iy}\cup \overline{jy}$ and $F_{y}^{(ij)}=(G_{y}-\overline{ij})\cup \overline{ix}\cup \overline{jx}$ of $K_{3,3,1,1}$ as illustrated in Fig. \ref{FxFy} (1) and (2), respectively. Since each of $F_{x}^{(ij)}$ and $F_{y}^{(ij)}$ is homeomorphic to $P_{7}$, by applying Theorem \ref{petersen_refine_pq} (2) to $f|_{F_{x}^{(ij)}}$, it follows that 
\begin{eqnarray}\label{fxij}
&&\sum_{\substack{
{\lambda=\gamma\cup \gamma'\in \Gamma_{3,5}^{(2)}(F_{x}^{(ij)})} 
\\ \gamma\in \Gamma_{3}(F_{x}^{(ij)}),\ \gamma'\in \Gamma_{5}(F_{x}^{(ij)})
\\ x\in \gamma,\ y\in \gamma'
}
}{\rm lk}(f(\lambda))^{2}
+\sum_{\substack{
{\lambda=\gamma\cup \gamma'\in \Gamma_{4,4}^{(2)}(F_{x}^{(ij)})} 
\\ x,y\in \gamma'
}
}{\rm lk}(f(\lambda))^{2}\\
&&+\sum_{\substack{
{\lambda=\gamma\cup \gamma'\in \Gamma_{3,4}^{(2)}(G_{x})} 
\\ \gamma\in \Gamma_{3}(G_{x}),\ \gamma'\in \Gamma_{4}(G_{x})
\\ \overline{ij}\not\subset\lambda,\ x\in \gamma
}
}{\rm lk}(f(\lambda))^{2}\nonumber\\
&=&
2\left\{
\sum_{\gamma\in \Gamma_{8}(F_{x}^{(ij)})}a_{2}(f(\gamma))
+\sum_{\substack{
{\gamma\in \Gamma_{7}(G_{x})} 
\\ \overline{ij}\not\subset\gamma
}
}a_{2}(f(\gamma))
\right\}\nonumber\\
&&-4\left\{
\sum_{\substack{
{\gamma\in \Gamma_{7}(F_{x}^{(ij)})} 
\\ x\not\in \gamma,\ y\in \gamma
}
}a_{2}(f(\gamma))
+\sum_{\substack{
{\gamma\in \Gamma_{6}(K_{3,3})} 
\\ \overline{ij}\not\subset\gamma
}
}a_{2}(f(\gamma))
\right\}\nonumber\\
&&
-2\left\{
\sum_{\substack{
{\gamma\in \Gamma_{6}(F_{x}^{(ij)})} 
\\ x,y\in \gamma
}
}a_{2}(f(\gamma))
+\sum_{\substack{
{\gamma\in \Gamma_{5}(G_{x})} 
\\ \overline{ij}\not\subset\gamma
}
}a_{2}(f(\gamma))
\right\}
+1. \nonumber
\end{eqnarray}

\begin{figure}[htbp]
      \begin{center} 
\scalebox{0.5}{\includegraphics*{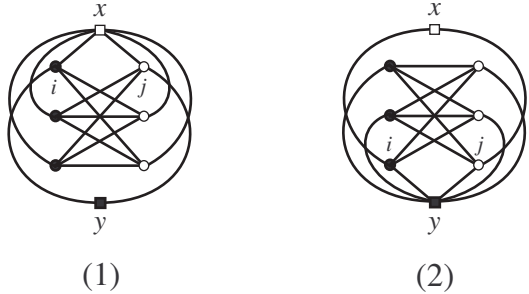}}
      \end{center}
   \caption{(1) $F_{x}^{(ij)}$, (2) $F_{y}^{(ij)}$ $(i=1,3,5,\ j=2,4,6)$}
  \label{FxFy}
\end{figure}

Let us take the sum of both sides of (\ref{fxij}) over $i=1,3,5$ and $j=2,4,6$. For an element $\gamma$ in $\Gamma_{8}(K_{3,3,1,1})$ of Type A, there uniquely exists $F_{x}^{(ij)}$ containing $\gamma$. This implies that  
\begin{eqnarray}\label{fxij1}
\sum_{i,j}\left(
\sum_{\gamma\in \Gamma_{8}(F_{x}^{(ij)})}a_{2}(f(\gamma))
\right)
=\sum_{\substack{
{\gamma\in \Gamma_{8}(K_{3,3,1,1})} 
\\ {\rm Type\ A} 
}
}a_{2}(f(\gamma)). 
\end{eqnarray}
For an element $\gamma$ of $\Gamma_{7}(G_{x})$, there exist exactly four edges of $K_{3,3}$ which are not contained in $\gamma$. Thus $\gamma$ is common for exactly four $F_{x}^{(ij)}$'s. This implies that 
\begin{eqnarray}\label{fxij2}
\sum_{i,j}
\left(
\sum_{\substack{
{\gamma\in \Gamma_{7}(G_{x})} 
\\ \overline{ij}\not\subset\gamma
}
}a_{2}(f(\gamma))
\right)
=4\sum_{\Gamma_{7}(G_{x})}a_{2}(f(\gamma)). 
\end{eqnarray}
For an element $\gamma$ in $\Gamma_{7}(G_{y})$, there uniquely exists $F_{x}^{(ij)}$ containing $\gamma$. This implies that 
\begin{eqnarray}\label{fxij3}
\sum_{i,j}
\left(
\sum_{\substack{
{\gamma\in \Gamma_{7}(F_{x}^{(ij)})} 
\\ x\not\in \gamma,\ y\in \gamma
}
}a_{2}(f(\gamma))
\right)
=\sum_{\gamma\in \Gamma_{7}(G_{y})}a_{2}(f(\gamma)). 
\end{eqnarray}
For an element $\gamma$ in $\Gamma_{6}(K_{3,3})$, there exist exactly three edges of $K_{3,3}$ which are not contained in $\gamma$. Thus $\gamma$ is common for exactly three $F_{x}^{(ij)}$'s. This implies that 
\begin{eqnarray}\label{fxij4}
\sum_{i,j}
\left(
\sum_{\substack{
{\gamma\in \Gamma_{6}(K_{3,3})} 
\\ \overline{ij}\not\subset\gamma
}
}a_{2}(f(\gamma))
\right)
=3\sum_{\gamma\in\Gamma_{6}(K_{3,3})}a_{2}(f(\gamma)). 
\end{eqnarray}
For an element $\gamma$ in $\Gamma_{6}(K_{3,3,1,1})$ containing $x$ and $y$, if $\gamma$ is of Type A, then there uniquely exists $F_{x}^{(ij)}$ containing $\gamma$. This implies that  
\begin{eqnarray}\label{fxij5}
\sum_{i,j}
\left(
\sum_{\substack{
{\gamma\in \Gamma_{6}(F_{x}^{(ij)})} 
\\ x,y\in \gamma
}
}a_{2}(f(\gamma))
\right)
=
\sum_{\substack{
{\gamma\in \Gamma_{6}(K_{3,3,1,1})} 
\\ x,y\in \gamma,\ {\rm Type\ A}
}
}a_{2}(f(\gamma)). 
\end{eqnarray}
For an element $\gamma$ in $\Gamma_{5}(G_{x})$, there exist exactly six edges of $K_{3,3}$ which are not contained in $\gamma$. Thus $\gamma$ is common for exactly six $F_{x}^{(ij)}$'s. This implies that 
\begin{eqnarray}\label{fxij6}
\sum_{i,j}
\left(
\sum_{\substack{
{\gamma\in \Gamma_{5}(G_{x})} 
\\ \overline{ij}\not\subset\gamma
}
}a_{2}(f(\gamma))
\right)
=
6\sum_{\gamma\in \Gamma_{5}(G_{x})}a_{2}(f(\gamma)). 
\end{eqnarray}

For an element $\lambda=\gamma\cup \gamma'$ in $\Gamma_{3,5}^{(2)}(K_{3,3,1,1})$ where $\gamma$ is a $3$-cycle and $\gamma'$ is a $5$-cycle, if $\gamma$ contains $x$ and $\gamma'$ contains $y$, then there uniquely exists $F_{x}^{(ij)}$ containing $\lambda$. This implies that 
\begin{eqnarray}\label{fxij7}
&&\sum_{i,j}
\left(
\sum_{\substack{
{\lambda=\gamma\cup \gamma'\in \Gamma_{3,5}^{(2)}(F_{x}^{(ij)})} 
\\ \gamma\in \Gamma_{3}(F_{x}^{(ij)}),\ \gamma'\in \Gamma_{5}(F_{x}^{(ij)})
\\ x\in \gamma,\ y\in \gamma'
}
}{\rm lk}(f(\lambda))^{2}
\right)\\
&=&
\sum_{\substack{
{\lambda=\gamma\cup \gamma'\in \Gamma_{3,5}^{(2)}(K_{3,3,1,1})} 
\\ \gamma\in \Gamma_{3}(K_{3,3,1,1}),\ \gamma'\in \Gamma_{5}(K_{3,3,1,1})
\\ x\in \gamma,\ y\in \gamma'
}
}{\rm lk}(f(\lambda))^{2}. \nonumber
\end{eqnarray}
For an element $\lambda$ in $\Gamma_{4,4}^{(2)}(K_{3,3,1,1})$ of Type A, there uniquely exists $F_{x}^{(ij)}$ containing $\lambda$. This implies that  
\begin{eqnarray}\label{fxij8}
\sum_{i,j}
\left(
\sum_{\substack{
{\lambda=\gamma\cup \gamma'\in \Gamma_{4,4}^{(2)}(F_{x}^{(ij)})} 
\\ x,y\in \gamma'
}
}{\rm lk}(f(\lambda))^{2}
\right)
=
\sum_{\substack{
{\lambda\in \Gamma_{4,4}^{(2)}(K_{3,3,1,1})} 
\\ {\rm Type\ A}
}
}{\rm lk}(f(\lambda))^{2}. 
\end{eqnarray}
For an element $\lambda$ in $\Gamma_{3,4}^{(2)}(G_{x})$, there exist exactly four edges of $K_{3,3}$ which are not contained in $\lambda$. Thus $\lambda$ is common for exactly four $F_{x}^{(ij)}$'s. This implies that 
\begin{eqnarray}\label{fxij9}
\sum_{i,j}
\left(
\sum_{\substack{
{\lambda=\gamma\cup \gamma'\in \Gamma_{3,4}^{(2)}(G_{x})} 
\\ \gamma\in \Gamma_{3}(G_{x}),\ \gamma'\in \Gamma_{4}(G_{x})
\\ \overline{ij}\not\subset\lambda,\ x\in \gamma
}
}{\rm lk}(f(\lambda))^{2}
\right)
=
4\sum_{\lambda\in \Gamma_{3,4}^{(2)}(G_{x})}{\rm lk}(f(\lambda))^{2}. 
\end{eqnarray}
Thus by (\ref{fxij}), (\ref{fxij1}), (\ref{fxij2}), (\ref{fxij3}), (\ref{fxij4}), (\ref{fxij5}), (\ref{fxij6}), (\ref{fxij7}), (\ref{fxij8}) and (\ref{fxij9}), we have 
\begin{eqnarray}\label{fxij10}
&&\sum_{\substack{
{\lambda=\gamma\cup \gamma'\in \Gamma_{3,5}^{(2)}(K_{3,3,1,1})} 
\\ \gamma\in \Gamma_{3}(K_{3,3,1,1}),\ \gamma'\in \Gamma_{5}(K_{3,3,1,1})
\\ x\in \gamma,\ y\in \gamma'
}
}{\rm lk}(f(\lambda))^{2}
+
\sum_{\substack{
{\lambda\in \Gamma_{4,4}^{(2)}(K_{3,3,1,1})} 
\\ {\rm Type\ A}
}
}{\rm lk}(f(\lambda))^{2} \\
&&+4\sum_{\lambda\in \Gamma_{3,4}^{(2)}(G_{x})}{\rm lk}(f(\lambda))^{2}\nonumber\\
&=&
2\sum_{\substack{
{\gamma\in \Gamma_{8}(K_{3,3,1,1})} 
\\ {\rm Type\ A}
}
}a_{2}(f(\gamma))
+8\sum_{\gamma\in \Gamma_{7}(G_{x})}a_{2}(f(\gamma))
-4\sum_{\gamma\in \Gamma_{7}(G_{y})}a_{2}(f(\gamma))\nonumber\\
&&-12\sum_{\gamma\in \Gamma_{6}(K_{3,3})}a_{2}(f(\gamma))
-2\sum_{\substack{
{\gamma\in \Gamma_{6}(K_{3,3,1,1})} 
\\ x,y\in \gamma,\ {\rm Type\ A}
}
}a_{2}(f(\gamma))\nonumber\\
&&-12\sum_{\gamma\in \Gamma_{5}(G_{x})}a_{2}(f(\gamma))
+9.  \nonumber
\end{eqnarray}
Then by combining (\ref{fxij10}) and (\ref{gx}), we have 
\begin{eqnarray}\label{fxij11}
&&\sum_{\substack{
{\lambda=\gamma\cup \gamma'\in \Gamma_{3,5}^{(2)}(K_{3,3,1,1})} 
\\ \gamma\in \Gamma_{3}(K_{3,3,1,1}),\ \gamma'\in \Gamma_{5}(K_{3,3,1,1})
\\ x\in \gamma,\ y\in \gamma'
}
}{\rm lk}(f(\lambda))^{2}
+
\sum_{\substack{
{\lambda\in \Gamma_{4,4}^{(2)}(K_{3,3,1,1})} 
\\ {\rm Type\ A}
}
}{\rm lk}(f(\lambda))^{2} \\
&=&
2\sum_{\substack{
{\gamma\in \Gamma_{8}(K_{3,3,1,1})} 
\\ {\rm Type\ A}
}
}a_{2}(f(\gamma))
-4\sum_{\gamma\in \Gamma_{7}(G_{y})}a_{2}(f(\gamma))
+4\sum_{\gamma\in \Gamma_{6}(K_{3,3})}a_{2}(f(\gamma))\nonumber\\
&&-2\sum_{\substack{
{\gamma\in \Gamma_{6}(K_{3,3,1,1})} 
\\ x,y\in \gamma,\ {\rm Type\ A}
}
}a_{2}(f(\gamma))
-4\sum_{\gamma\in \Gamma_{5}(G_{x})}a_{2}(f(\gamma))
+5. \nonumber
\end{eqnarray}

By applying Theorem \ref{petersen_refine_pq} (2) to $f|_{F_{y}^{(ij)}}$ and combining the same argument as in the case of $F_{x}^{(ij)}$ with (\ref{gy}), we also have 
\begin{eqnarray}\label{fxij12}
&&\sum_{\substack{
{\lambda=\gamma\cup \gamma'\in \Gamma_{3,5}^{(2)}(K_{3,3,1,1})} 
\\ \gamma\in \Gamma_{3}(K_{3,3,1,1}),\ \gamma'\in \Gamma_{5}(K_{3,3,1,1})
\\ y\in \gamma,\ x\in \gamma'
}
}{\rm lk}(f(\lambda))^{2}
+
\sum_{\substack{
{\lambda\in \Gamma_{4,4}^{(2)}(K_{3,3,1,1})} 
\\ {\rm Type\ A}
}
}{\rm lk}(f(\lambda))^{2} \\
&=&
2\sum_{\substack{
{\gamma\in \Gamma_{8}(K_{3,3,1,1})} 
\\ {\rm Type\ A}
}
}a_{2}(f(\gamma))
-4\sum_{\gamma\in \Gamma_{7}(G_{x})}a_{2}(f(\gamma))
+4\sum_{\gamma\in \Gamma_{6}(K_{3,3})}a_{2}(f(\gamma))\nonumber\\
&&-2\sum_{\substack{
{\gamma\in \Gamma_{6}(K_{3,3,1,1})} 
\\ y,x\in \gamma,\ {\rm Type\ A}
}
}a_{2}(f(\gamma))
-4\sum_{\gamma\in \Gamma_{5}(G_{y})}a_{2}(f(\gamma))
+5.  \nonumber
\end{eqnarray}
Then by adding (\ref{fxij11}) and (\ref{fxij12}), we have the result. 
\end{proof}

\begin{Lemma}\label{l2}
For any element $f$ in ${\rm SE}(K_{3,3,1,1})$, 
\begin{eqnarray}\label{q82}
&&\sum_{\substack{
{\lambda\in \Gamma_{4,4}^{(2)}(K_{3,3,1,1})} 
\\ {\rm Type\ B}
}
}{\rm lk}(f(\lambda))^{2}\\
&=&
2\sum_{\substack{
{\gamma\in \Gamma_{8}(K_{3,3,1,1})} 
\\ {\rm Type\ B}
}
}a_{2}(f(\gamma))
+4\sum_{\gamma\in \Gamma_{6}(K_{3,3})}a_{2}(f(\gamma))\nonumber \\
&&-2\left\{\sum_{\substack{
{\gamma\in \Gamma_{6}(G_{x})} 
\\ x\in \gamma
}
}
a_{2}(f(\gamma)) 
+\sum_{\substack{
{\gamma\in \Gamma_{6}(G_{y})} 
\\ y\in \gamma
}
}
a_{2}(f(\gamma))\right\} \nonumber\\
&&-2\sum_{\substack{
{\gamma\in \Gamma_{6}(K_{3,3,1,1})} 
\\ x,y\in \gamma,\ {\rm Type B}
}
}a_{2}(f(\gamma))
+2.  \nonumber
\end{eqnarray}
\end{Lemma}

\begin{proof}
Let us consider subgraphs $Q_{8}^{(1)}=K_{3,3}\cup \overline{x1}\cup \overline{x3}\cup \overline{x5}\cup \overline{y2}\cup \overline{y4}\cup \overline{y6}$ and $Q_{8}^{(2)}=K_{3,3}\cup \overline{x2}\cup \overline{x4}\cup \overline{x6}\cup \overline{y1}\cup \overline{y3}\cup \overline{y5}$ of $K_{3,3,1,1}$ as illustrated in Fig. \ref{Q8_12} (1) and (2), respectively. Since each of $Q_{8}^{(1)}$ and $Q_{8}^{(2)}$ is homeomorphic to $Q_{8}$, by applying Theorem \ref{petersen_refine_pq} (1) to $f|_{Q_{8}^{(1)}}$ and $f|_{Q_{8}^{(2)}}$, it follows that 
\begin{eqnarray}\label{q81}
\sum_{\lambda\in \Gamma_{4,4}^{(2)}(Q_{8}^{(i)})}{\rm lk}(f(\lambda))^{2}
&=&
2\sum_{\gamma\in \Gamma_{8}(Q_{8}^{(i)})}a_{2}(f(\gamma))
+2\sum_{\gamma\in \Gamma_{6}(K_{3,3})}a_{2}(f(\gamma))\\
&&-2\sum_{\substack{
{\gamma\in \Gamma_{6}(Q_{8}^{(i)})} 
\\ x\in \gamma,\ y\not\in \gamma
}
}a_{2}(f(\gamma))
-2\sum_{\substack{
{\gamma\in \Gamma_{6}(Q_{8}^{(i)})} 
\\ x\not\in \gamma,\ y\in \gamma
}
}a_{2}(f(\gamma))\nonumber\\
&&-2\sum_{\substack{
{\gamma\in \Gamma_{6}(Q_{8}^{(i)})} 
\\ x,y\in \gamma
}
}a_{2}(f(\gamma))
+1 \nonumber
\end{eqnarray}
for $i=1,2$. By adding (\ref{q81}) for $i=1,2$, we have the result. 
\end{proof}

\begin{figure}[htbp]
      \begin{center} 
\scalebox{0.5}{\includegraphics*{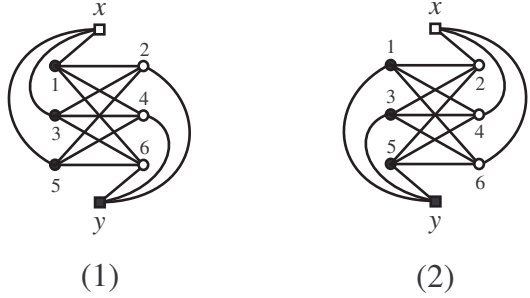}}
      \end{center}
   \caption{(1) $Q_{8}^{(1)}$, (2) $Q_{8}^{(2)}$}
  \label{Q8_12}
\end{figure}

\begin{Lemma}\label{l3}
For any element $f$ in ${\rm SE}(K_{3,3,1,1})$,  
\begin{eqnarray}\label{fxk9}
&&\sum_{\substack{
{\lambda\in \Gamma_{4,4}^{(2)}(K_{3,3,1,1})} 
\\ {\rm Type\ C}
}
}{\rm lk}(f(\lambda))^{2}\\
&=&
2\sum_{\substack{
{\gamma\in \Gamma_{8}(K_{3,3,1,1})} 
\\ {\rm Type\ C}
}
}a_{2}(f(\gamma))-8\sum_{\gamma\in \Gamma_{6}(K_{3,3})}a_{2}(f(\gamma))\nonumber \\
&&-2\sum_{\substack{
{\gamma\in \Gamma_{6}(K_{3,3,1,1})} 
\\ x,y\in \gamma,\ {\rm Type\ C}
}
}a_{2}(f(\gamma))
+2. \nonumber
\end{eqnarray}
\end{Lemma}

\begin{proof}
For $k=1,2,\ldots,6$, let us consider subgraphs $F_{x}^{(k)}=(G_{x}-\overline{xk})\cup \overline{xy}\cup \overline{ky}$ and $F_{y}^{(k)}=(G_{y}-\overline{yk})\cup \overline{kx}\cup \overline{yx}$ of $K_{3,3,1,1}$ as illustrated in Fig. \ref{FxFy2} (1) and (2), respectively. Since each of $F_{x}^{(k)}$ and $F_{y}^{(k)}$ is also homeomorphic to $P_{7}$, by applying Theorem \ref{petersen_refine_pq} (2) to $f|_{F_{x}^{(k)}}$, it follows that 
\begin{eqnarray}\label{fxk}
&&\sum_{\substack{
{\lambda=\gamma\cup \gamma'\in \Gamma_{4,4}^{(2)}(F_{x}^{(ij)})} 
\\ x,y\in \gamma,\ {\rm Type\ C}
}
}{\rm lk}(f(\lambda))^{2}
+\sum_{\substack{
{\lambda\in \Gamma_{3,4}^{(2)}(G_{x})} 
\\ \overline{xk}\not\subset\lambda
}
}{\rm lk}(f(\lambda))^{2} \\
&=&
2\left\{
\sum_{\gamma\in \Gamma_{8}(F_{x}^{(k)})}a_{2}(f(\gamma))
+\sum_{\substack{
{\gamma\in \Gamma_{7}(G_{x})} 
\\ \overline{xk}\not\subset\gamma
}
}a_{2}(f(\gamma))
\right\}
-4\sum_{\gamma\in \Gamma_{6}(K_{3,3})}a_{2}(f(\gamma))\nonumber\\
&&
-2\left\{
\sum_{\substack{
{\gamma\in \Gamma_{6}(F_{x}^{(k)})} 
\\ x,y\in \gamma
}
}a_{2}(f(\gamma))
+\sum_{\substack{
{\gamma\in \Gamma_{5}(G_{x})} 
\\ \overline{xk}\not\subset\gamma
}
}a_{2}(f(\gamma))
\right\}
+1. \nonumber
\end{eqnarray}

\begin{figure}[htbp]
      \begin{center} 
\scalebox{0.5}{\includegraphics*{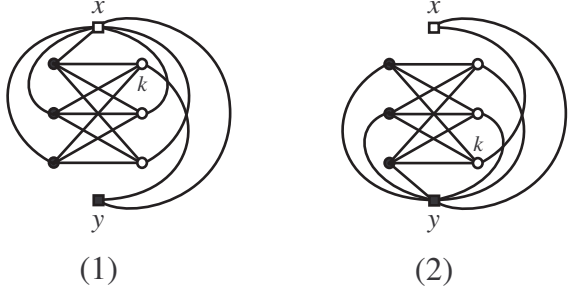}}
      \end{center}
   \caption{(1) $F_{x}^{(k)}$, (2) $F_{y}^{(k)}$ $(k=1,2,3,4,5,6)$}
  \label{FxFy2}
\end{figure}

Let us take the sum of both sides of (\ref{fxk}) over $k=1,2,\ldots,6$. For an element $\gamma$ in $\Gamma_{8}(K_{3,3,1,1})$, if $\gamma$ is of Type C, then there uniquely exists $F_{x}^{(k)}$ containing $\gamma$. This implies that  
\begin{eqnarray}\label{fxk1}
\sum_{k}\left(
\sum_{\gamma\in \Gamma_{8}(F_{x}^{(k)})}a_{2}(f(\gamma))
\right)
=\sum_{\substack{
{\gamma\in \Gamma_{8}(K_{3,3,1,1})} 
\\ {\rm Type\ C} 
}
}a_{2}(f(\gamma)). 
\end{eqnarray}
For an element $\gamma$ of $\Gamma_{7}(G_{x})$, there exist exactly four edges which are incident to $x$ such that they are not contained in $\gamma$. Thus $\gamma$ is common for exactly four $F_{x}^{(k)}$'s. This implies that 
\begin{eqnarray}\label{fxk2}
\sum_{k}
\left(
\sum_{\substack{
{\gamma\in \Gamma_{7}(G_{x})} 
\\ \overline{xk}\not\subset\gamma
}
}a_{2}(f(\gamma)). 
\right)
=4\sum_{\Gamma_{7}(G_{x})}a_{2}(f(\gamma)). 
\end{eqnarray}
It is clear that any element $\gamma$ in $\Gamma_{6}(K_{3,3})$ is common for exactly six $F_{x}^{(k)}$'s. This implies that 
\begin{eqnarray}\label{fxk3}
\sum_{k}
\left(
\sum_{\gamma\in \Gamma_{6}(K_{3,3})}a_{2}(f(\gamma))
\right)
=6\sum_{\gamma\in\Gamma_{6}(K_{3,3})}a_{2}(f(\gamma)).
\end{eqnarray}
For an element $\gamma$ in $\Gamma_{6}(K_{3,3,1,1})$ containing $x$ and $y$, if $\gamma$ is of Type C, then there uniquely exists $F_{x}^{(k)}$ containing $\gamma$. This implies that  
\begin{eqnarray}\label{fxk4}
\sum_{k}
\left(
\sum_{\substack{
{\gamma\in \Gamma_{6}(F_{x}^{(k)})} 
\\ x,y\in \gamma
}
}a_{2}(f(\gamma))
\right)
=
\sum_{\substack{
{\gamma\in \Gamma_{6}(K_{3,3,1,1})} 
\\ x,y\in \gamma,\ {\rm Type\ C}
}
}a_{2}(f(\gamma)). 
\end{eqnarray}
For an element $\gamma$ of $\Gamma_{5}(G_{x})$, there exist exactly four edges which are incident to $x$ such that they are not contained in $\gamma$. Thus $\gamma$ is common for exactly four $F_{x}^{(k)}$'s. This implies that 
\begin{eqnarray}\label{fxk5}
\sum_{k}
\left(
\sum_{\substack{
{\gamma\in \Gamma_{5}(G_{x})} 
\\ \overline{xk}\not\subset\gamma
}
}a_{2}(f(\gamma))
\right)
=
4\sum_{\gamma\in \Gamma_{5}(G_{x})}a_{2}(f(\gamma)). 
\end{eqnarray}
For an element $\lambda=\gamma\cup \gamma'$ in $\Gamma_{4,4}^{(2)}(K_{3,3,1,1})$, if $\lambda$ is of Type C, then there uniquely exists $F_{x}^{(k)}$ containing $\lambda$. This implies that 
\begin{eqnarray}\label{fxk6}
\sum_{k}
\left(
\sum_{\substack{
{\lambda=\gamma\cup \gamma'\in \Gamma_{4,4}^{(2)}(F_{x}^{(k)})} 
\\ x,y\in \gamma,\ {\rm Type C}
}
}{\rm lk}(f(\lambda))^{2}
\right)
=
\sum_{\substack{
{\lambda\in \Gamma_{4,4}^{(2)}(K_{3,3,1,1})} 
\\ {\rm Type\ C}
}
}{\rm lk}(f(\lambda))^{2}. 
\end{eqnarray}
For an element $\lambda$ in $\Gamma_{3,4}^{(2)}(G_{x})$, there exist exactly four edges which are incident to $x$ such that they are not contained in $\lambda$. Thus $\lambda$ is common for exactly four $F_{x}^{(k)}$'s. This implies that 
\begin{eqnarray}\label{fxk7}
\sum_{k}
\left(
\sum_{\substack{
{\lambda\in \Gamma_{3,4}^{(2)}(G_{x})} 
\\ \overline{xk}\not\subset\lambda
}
}{\rm lk}(f(\lambda))^{2}
\right)
=
4\sum_{\lambda\in \Gamma_{3,4}^{(2)}(G_{x})}{\rm lk}(f(\lambda))^{2}. 
\end{eqnarray}
Then by (\ref{fxk}), (\ref{fxk1}), (\ref{fxk2}), (\ref{fxk3}), (\ref{fxk4}), (\ref{fxk5}), (\ref{fxk6}) and (\ref{fxk7}), we have 
\begin{eqnarray}\label{fxk8}
&&
\sum_{\substack{
{\lambda\in \Gamma_{4,4}^{(2)}(K_{3,3,1,1})} 
\\ {\rm Type\ C}
}
}{\rm lk}(f(\lambda))^{2}
+
4\sum_{\lambda\in \Gamma_{3,4}^{(2)}(G_{x})}{\rm lk}(f(\lambda))^{2} \\
&=&
2\sum_{\substack{
{\gamma\in \Gamma_{8}(K_{3,3,1,1})} 
\\ {\rm Type\ C}
}
}a_{2}(f(\gamma))
+8\sum_{\gamma\in \Gamma_{7}(G_{x})}a_{2}(f(\gamma))
-24\sum_{\gamma\in \Gamma_{6}(K_{3,3})}a_{2}(f(\gamma))\nonumber\\
&&-2\sum_{\substack{
{\gamma\in \Gamma_{6}(K_{3,3,1,1})} 
\\ x,y\in \gamma,\ {\rm Type\ C}
}
}a_{2}(f(\gamma))
-8\sum_{\gamma\in \Gamma_{5}(G_{x})}a_{2}(f(\gamma))
+6.  \nonumber
\end{eqnarray}
Then by combining (\ref{fxk8}) and (\ref{gx}), we have the reslut. We remark here that by by applying Theorem \ref{petersen_refine_pq} (2) to $f|_{F_{y}^{(k)}}$ combining the same argument as in the case of $F_{x}^{(k)}$ with (\ref{gy}), we also have (\ref{fxk9}). 
\end{proof}

\begin{proof}[Proof of Theorem \ref{main_theorem_k3311}.] 
(1) Let $f$ be an element in ${\rm SE}(K_{3,3,1,1})$. Then by combining (\ref{fxij13}), (\ref{q82}) and (\ref{fxk9}), we have 
\begin{eqnarray}\label{total1}
&&\sum_{\lambda\in \Gamma_{3,5}^{(2)}(K_{3,3,1,1})}{\rm lk}(f(\lambda))^{2}
+2\sum_{\lambda\in \Gamma_{4,4}^{(2)}(K_{3,3,1,1})}{\rm lk}(f(\lambda))^{2} \\
&=&
4\sum_{\gamma\in \Gamma_{8}(K_{3,3,1,1})}a_{2}(f(\gamma))
-4\left\{
\sum_{\gamma\in \Gamma_{7}(G_{x})}a_{2}(f(\gamma))
+\sum_{\gamma\in \Gamma_{7}(G_{y})}a_{2}(f(\gamma))
\right\}\nonumber\\
&&-4\left\{
\sum_{\substack{
{\gamma\in \Gamma_{6}(G_{x})} 
\\ x\in \gamma
}
}a_{2}(f(\gamma))
+\sum_{\substack{
{\gamma\in \Gamma_{6}(G_{y})} 
\\ y\in \gamma
}
}a_{2}(f(\gamma))
+
\sum_{\substack{
{\gamma\in \Gamma_{6}(K_{3,3,1,1})} 
\\ x,y\in \gamma
\\ {\rm Type\ A,B,C}
}
}
a_{2}(f(\gamma))\right\} \nonumber\\
&&-4\left\{
\sum_{\gamma\in \Gamma_{5}(G_{x})}a_{2}(f(\gamma))
+\sum_{\gamma\in \Gamma_{5}(G_{y})}a_{2}(f(\gamma))
\right\}
+18. \nonumber
\end{eqnarray}
Note that 
\begin{eqnarray*}
\Gamma_{k}(G_{x})\cup \Gamma_{k}(G_{y})=\left\{\gamma\in \Gamma_{k}(K_{3,3,1,1})\ |\ \left\{x,y\right\}\not\subset\gamma\right\}
\end{eqnarray*}
 for $k=5,7$. Moreover, we define a subset $\Gamma_{6}'$ of $\Gamma_{6}(K_{3,3,1,1})$ by 
\begin{eqnarray}\label{gamma6}
\Gamma_{6}'&=&\left\{\gamma\in \Gamma_{6}(G_{x})\ |\ x\in\gamma\right\}\cup \left\{\gamma\in \Gamma_{6}(G_{y})\ |\ y\in\gamma\right\}\\
&&\cup\left\{\gamma\in \Gamma_{6}(K_{3,3,1,1})\ |\ x,y\in \gamma,\ \gamma{\rm\ is\ of\ Type\ A,B\ or\ C}\right\}. \nonumber
\end{eqnarray}
Then we see that (\ref{total1}) implies (\ref{maineq}).

(2) Let $f$ be an element in ${\rm SE}(K_{3,3,1,1})$. Let us consider subgraphs $H_{1}=Q_{8}^{(1)}\cup \overline{xy}$ and $H_{2}=Q_{8}^{(2)}\cup \overline{xy}$ of $K_{3,3,1,1}$ as illustrated in Fig. \ref{H1H2} (1) and (2), respectively. For $i=1,2$, $H_{i}$ has the proper minor $H'_{i}=H_{i}/\overline{xy}$ which is isomorphic to $P_{7}$. For a spatial embedding $f|_{H_{i}}$ of $H_{i}$, there exists a spatial embedding $f'$ of $H'_{i}$ such that $f'(H'_{i})$ is obtained from $f(H_{i})$ by contracting $f(\overline{xy})$ into one point. Note that this embedding is unique up to ambient isotopy in ${\mathbb R}^{3}$. Then by Corollary \ref{CG_refine_Do}, there exists an element $\mu'_{i}$ in $\Gamma_{3,4}^{(2)}(H'_{i})$ such that ${\rm lk}(f'(\mu'_{i}))\equiv 1\pmod{2}\ (i=1,2)$. Note that $\mu'_{i}$ is mapped onto an element $\mu_{i}$ in $\Gamma_{4,4}(H_{i})$ by the natural injection from $\Gamma_{3,4}(H'_{i})$ to $\Gamma_{4,4}(H_{i})$. Since $f'(\mu'_{i})$ is ambient isotopic to $f(\mu_{i})$, we have ${\rm lk}(f(\mu_{i}))\equiv 1\pmod{2}\ (i=1,2)$. We also note that both $\mu_{1}$ and $\mu_{2}$ are of Type C in $\Gamma_{4,4}^{(2)}(K_{3,3,1,1})$. 

\begin{figure}[htbp]
      \begin{center} 
\scalebox{0.5}{\includegraphics*{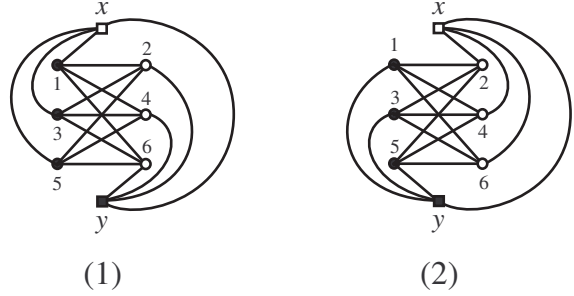}}
      \end{center}
   \caption{(1) $H_{1}$, (2) $H_{2}$}
  \label{H1H2}
\end{figure}

For $v=x,y$ and $i,j,k=1,2,\ldots,6\ (i\neq j)$, let $P_{8}^{(k)}(v;ij)$ be the subgraph of $K_{3,3,1,1}$ as illustrated in Fig. \ref{P8_1234} (1) if $v=y$, $k\in \left\{1,3,5\right\}$ and $i,j\in\left\{2,4,6\right\}$, (2) if $v=y$, $k\in\left\{2,4,6\right\}$ and $i,j\in\left\{1,3,5\right\}$, (3) if $v=x$, $k\in \left\{1,3,5\right\}$ and $i,j\in\left\{2,4,6\right\}$ and (4) if $v=x$, $k\in \left\{2,4,6\right\}$ and $i,j\in\left\{1,3,5\right\}$. Note that there exist exactly thirty six $P_{8}^{(k)}(v;ij)$'s and they are isomorphic to $P_{8}$ in the $K_{6}$-family. Thus by Corollary \ref{CG_refine_Do}, there exists an element $\lambda$ in $\Gamma^{(2)}(P_{8}^{(k)}(v;ij))$ such that ${\rm lk}(f(\lambda))\equiv 1\pmod{2}$. All elements in $\Gamma^{(2)}(P_{8}^{(k)}(v;ij))$ consist of exactly four elements in $\Gamma^{(2)}_{3,5}(P_{8}^{(k)}(v;ij))$ and exactly four elements in $\Gamma^{(2)}_{4,4}(P_{8}^{(k)}(v;ij))$ of Type A or Type B because they do not contain the edge $\overline{xy}$. It is not hard to see that any element in $\Gamma_{3,5}^{(2)}(K_{3,3,1,1})$ is common for exactly two $P_{8}^{(k)}(v;ij)$'s, and any element in $\Gamma_{4,4}^{(2)}(K_{3,3,1,1})$ of Type A or Type B is common for exactly four $P_{8}^{(k)}(v;ij)$'s. 

\begin{figure}[htbp]
      \begin{center} 
\scalebox{0.5}{\includegraphics*{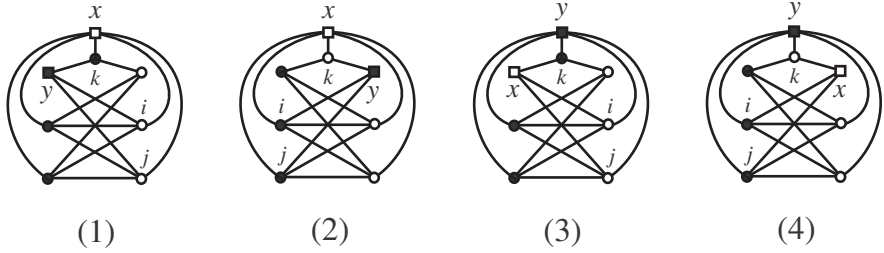}}
      \end{center}
   \caption{$P_{8}^{(k)}(v;ij)$}
  \label{P8_1234}
\end{figure}

By (\ref{fxij13}), there exist a nonnegative integer $m$ such that 
\begin{eqnarray*}
\sum_{\lambda\in \Gamma_{3,5}^{(2)}(K_{3,3,1,1})}{\rm lk}(f(\lambda))^{2}=2m. 
\end{eqnarray*}
If $2m\ge 18$, since there exist at least two elements $\mu_{1}$ and $\mu_{2}$ in $\Gamma_{4,4}^{(2)}(K_{3,3,1,1})$ of Type C such that ${\rm lk}(f(\mu_{i}))\equiv 1\pmod{2}\ (i=1,2)$, we have  
\begin{eqnarray*}
\sum_{\lambda\in \Gamma_{3,5}^{(2)}(K_{3,3,1,1})}{\rm lk}(f(\lambda))^{2}+2\sum_{\lambda\in \Gamma_{4,4}^{(2)}(K_{3,3,1,1})}{\rm lk}(f(\lambda))^{2}
\ge 18+4=22. 
\end{eqnarray*}
If $2m\le 16$, then there exist at least $(36-4m)/4=9-m$ elements in $\Gamma_{4,4}^{(2)}(K_{3,3,1,1})$ of Type A or Type B such that each of the corresponding $2$-component links with respect to $f$ has an odd linking number. Then we have 
\begin{eqnarray*}
\sum_{\lambda\in \Gamma_{3,5}^{(2)}(K_{3,3,1,1})}{\rm lk}(f(\lambda))^{2}+2\sum_{\lambda\in \Gamma_{4,4}^{(2)}(K_{3,3,1,1})}{\rm lk}(f(\lambda))^{2}
\ge 2m+2\left\{(9-m)+2\right\}=22. 
\end{eqnarray*}
This completes the proof. 

\end{proof}

\section{$\triangle Y$-exchange and Conway-Gordon type formulae} 

In this section, we give a proof of Theorem \ref{NT_main_cor}. Let $G_{\triangle}$ and $G_{Y}$ be two graphs such that $G_{Y}$ is obtained from $G_{\triangle}$ by a single $\triangle Y$-exchange. Let $\gamma'$ be an element in $\Gamma(G_{\triangle})$ which does not contain $\triangle$. Then there exists an element $\Phi(\gamma')$ in $\Gamma(G_{Y})$ such that $\gamma'\setminus \triangle=\Phi(\gamma')\setminus Y$. It is easy to see that the correspondence from $\gamma'$ to $\Phi(\gamma')$ defines a surjective map 
\begin{eqnarray*}\label{phi}
\Phi:\Gamma(G_{\triangle})\setminus \left\{\triangle\right\}\longrightarrow \Gamma(G_{Y}).
\end{eqnarray*}
The inverse image of an element $\gamma$ in $\Gamma(G_{Y})$ by $\Phi$ contains at most two elements in $\Gamma(G_{\triangle})\setminus \Gamma_{\triangle}(G_{\triangle})$. Fig. \ref{not_inj3} illustrates the case that the inverse image of $\gamma$ by $\Phi$ consists of exactly two elements. Let $\omega$ be a map from $\Gamma(G_{\triangle})$ to ${\mathbb Z}$. Then we define the map $\tilde{\omega}$ from $\Gamma(G_{Y})$ to ${\mathbb Z}$ by 
\begin{eqnarray}\label{omega_map}
\tilde{\omega}(\gamma)=\sum_{\gamma'\in \Phi^{-1}(\gamma)}\omega(\gamma')
\end{eqnarray}
for an element $\gamma$ in $\Gamma(G_{Y})$. 

\begin{figure}[htbp]
      \begin{center}
\scalebox{0.425}{\includegraphics*{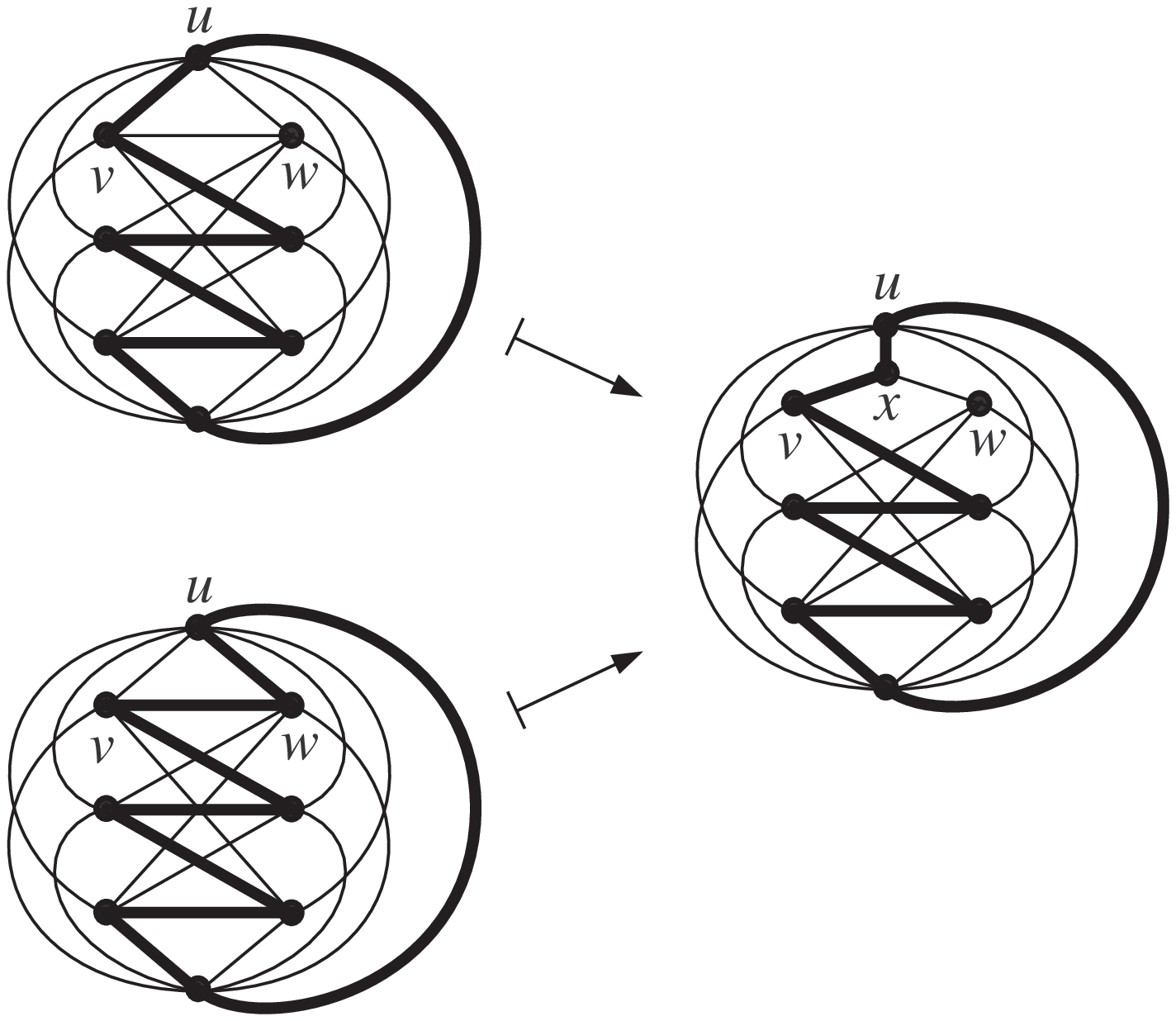}}
      \end{center}
   \caption{}
  \label{not_inj3}
\end{figure} 

Let $f$ be an element in ${\rm SE}(G_{Y})$ and $D$ a $2$-disk in ${\mathbb R}^{3}$ such that $D\cap f(G_{Y})=f(Y)$ and $\partial D \cap f(G_{Y}) = \{f(u),f(v),f(w)\}$. Let $\varphi(f)$ be an element in ${\rm SE}(G_{\triangle})$ such that $\varphi(f)(x)=f(x)$ for $x\in G_{\triangle}\setminus \triangle = G_{Y}\setminus Y$ and $\varphi(f)(G_{\triangle})=\left(f(G_{Y})\setminus f(Y)\right)\cup \partial D$. Thus we obtain a map 
\begin{eqnarray*}\label{varphi}
\varphi:{\rm SE}(G_{Y})\longrightarrow {\rm SE}(G_{\triangle}). 
\end{eqnarray*}
Then we immediately have the following. 
\begin{Proposition}\label{map} 
Let $f$ be an element in ${\rm SE}(G_{Y})$ and $\gamma$ an element in $\Gamma(G_{Y})$. Then, $f(\gamma)$ is ambient isotopic to $\varphi(f)(\gamma')$ for each element $\gamma'$ in the inverse image of $\gamma$ by $\Phi$. 
\end{Proposition}

Then we have the following lemma which plays a key role to prove Theorem \ref{NT_main_cor}. This lemma has already been shown in \cite[Lemma 2.2]{NT12} in more general form, but we give a proof for the reader's convenience. 

\begin{Lemma}\label{lemma2a} 
{\rm (Nikkuni-Taniyama \cite{NT12})} 
For an element $f$ in ${\rm SE}(G_{Y})$,  
\begin{eqnarray*}
\sum_{\gamma\in \Gamma(G_{Y})}\tilde{\omega}(\gamma)a_{2}(f(\gamma))
=\sum_{\gamma'\in \Gamma(G_{\triangle})}\omega(\gamma')a_{2}(\varphi(f)(\gamma')).
\end{eqnarray*}
\end{Lemma}

\begin{proof}
Since $\varphi(f)(\triangle)$ is the trivial knot, we have  
\begin{eqnarray*}
\sum_{\gamma'\in\Gamma(G_\triangle)}\omega(\gamma')a_{2}(\varphi(f)(\gamma'))
= \sum_{\gamma'\in\Gamma(G_{\triangle})\setminus \left\{\triangle\right\}}\omega(\gamma')a_{2}(\varphi(f)(\gamma')).
\end{eqnarray*}
Note that 
\begin{eqnarray*}
\Gamma(G_{\triangle})\setminus \left\{\triangle\right\}
=\bigcup_{\gamma\in\Gamma(G_{Y})}\Phi^{-1}(\gamma). 
\end{eqnarray*}
Then, by Proposition \ref{map}, we see that 
\begin{eqnarray*}
\sum_{\gamma'\in\Gamma(G_{\triangle})\setminus \left\{\triangle\right\}}\omega(\gamma')a_{2}(\varphi(f)(\gamma'))
&=&\sum_{\gamma\in\Gamma(G_Y)}\left(\sum_{\gamma'\in\Phi^{-1}(\gamma)}\omega(\gamma')a_{2}(\varphi(f)(\gamma'))\right)\\
&=&\sum_{\gamma\in\Gamma(G_Y)}\left(\sum_{\gamma'\in\Phi^{-1}(\gamma)}\omega(\gamma)a_{2}(f(\gamma))\right)
\\
&=& \sum_{\gamma\in\Gamma(G_{Y})}\tilde{\omega}(\gamma)a_{2}(f(\gamma)). 
\end{eqnarray*}
Thus we have the result. 
\end{proof}

\begin{proof}[Proof of Theorem \ref{NT_main_cor}.] 
By Corollary \ref{main_theorem_k3311_cor}, there exists a map $\omega$ from $\Gamma(K_{3,3,1,1})$ to ${\mathbb Z}$ such that for any element $g$ in ${\rm SE}(K_{3,3,1,1})$,  
\begin{eqnarray}\label{g}
\sum_{\gamma'\in \Gamma(K_{3,3,1,1})}\omega(\gamma')a_{2}(g(\gamma'))
\ge 1.
\end{eqnarray}
Let $G$ be a graph which is obtained from $K_{3,3,1,1}$ by a single $\triangle Y$-exchange and $\tilde{\omega}$ the map from $\Gamma(G)$ to ${\mathbb Z}$ as in (\ref{omega_map}). Let $f$ be an element in ${\rm SE}(G)$. Then by Lemma \ref{lemma2a} and (\ref{g}), we see that 
\begin{eqnarray*}
\sum_{\gamma\in \Gamma(G)}\tilde{\omega}(\gamma)a_{2}(f(\gamma))
=\sum_{\gamma'\in \Gamma(K_{3,3,1,1})}\omega(\gamma')a_{2}(\varphi(f)(\gamma'))
\ge 1.
\end{eqnarray*}
By repeating this argument, we have the result. 
\end{proof}

\begin{Remark}\label{expl}
In Theorem \ref{NT_main_cor}, the proof of the existence of a map $\omega$ is constructive. It is also an interesting problem to give $\omega(\gamma)$ for each element $\gamma$ in ${\Gamma}(G)$ concretely. 
\end{Remark}

\section{Rectilinear spatial embeddings of $K_{3,3,1,1}$} 

In this section, we give a proof of Theorem \ref{main_theorem_recti}. For an element $f$ in ${\rm RSE}(G)$ and an element $\gamma$ in $\Gamma_{k}(G)$, the knot $f(\gamma)$ has {\it stick number} less than or equal to $k$, where the stick number $s(K)$ of a knot $K$ is the minimum number of edges in a polygon which represents $K$. Then the following is well known. 

\begin{Proposition}\label{stick}{\rm (Adams \cite{adams}, Negami \cite{Ne91})} 
For any nontrivial knot $K$, it follows that $s(K)\ge 6$. Moreover, $s(K)=6$ if and only if $K$ is a trefoil knot. 
\end{Proposition}

We also show a lemma for a rectilinear spatial embedding of $P_{7}$ which is useful in proving Theorem \ref{main_theorem_recti}. 

\begin{Lemma}\label{recti_P7}
For an element $f$ in ${\rm RSE}\left(P_{7}\right)$,  
\begin{eqnarray*}
\sum_{\gamma\in \Gamma_{7}\left(P_{7}\right)}a_{2}(f(\gamma))\ge 0. 
\end{eqnarray*}
\end{Lemma}

\begin{proof}
Note that $a_{2}({\rm trivial\ knot})=0$ and $a_{2}({\rm trefoil\ knot})=1$. Thus by Proposition \ref{stick}, $a_{2}(f(\gamma))=0$ for any element $\gamma$ in $\Gamma_{5}(P_{7})$ and $a_{2}(f(\gamma))\ge 0$ for any element $\gamma$ in $\Gamma_{6}(P_{7})$. Moreover, by Corollary \ref{CG_refine_Do}, we have 
\begin{eqnarray}\label{34lk}
\sum_{\lambda\in \Gamma_{3,4}^{(2)}(P_{7})}{\rm lk}(f(\lambda))^{2}\ge 1. 
\end{eqnarray}
Then Theorem \ref{petersen_refine_pq} (2) implies the result. 
\end{proof}

\begin{proof}[Proof of Theorem \ref{main_theorem_recti}.] 
Let $f$ be an element in ${\rm RSE(K_{3,3,1,1})}$. Since $G_{x}$ and $G_{y}$ are isomorphic to $P_{7}$, by Lemma \ref{recti_P7}, we have 
\begin{eqnarray}\label{34a2}
\sum_{\gamma\in \Gamma_{7}\left(G_{x}\right)}a_{2}(f(\gamma)) \ge 0,\ 
\sum_{\gamma\in \Gamma_{7}\left(G_{y}\right)}a_{2}(f(\gamma)) \ge 0. 
\end{eqnarray}
Then by Corollary \ref{main_theorem_k3311_cor} and (\ref{34a2}), we have 
\begin{eqnarray*}\label{k3311_ineq2}
\sum_{\gamma\in \Gamma_{8}(K_{3,3,1,1})}a_{2}(f(\gamma))
&\ge & 
\sum_{\gamma\in \Gamma_{7}\left(G_{x}\right)}a_{2}(f(\gamma)) 
+ \sum_{\gamma\in \Gamma_{7}\left(G_{y}\right)}a_{2}(f(\gamma)) \\
&& +\sum_{\gamma\in \Gamma_{6}'}a_{2}(f(\gamma)) 
+\sum_{\substack{{\gamma\in \Gamma_{5}(K_{3,3,1,1})} \\ {\left\{x,y\right\}\not\subset\gamma }}}a_{2}(f(\gamma))
+1 \\
&\ge& 0+0+0+0+1 \\
&=& 1. 
\end{eqnarray*}
Thus we have the desired conclusion. 
\end{proof}

\begin{Remark}
All of knots with $s\le 8$ and $a_{2}>0$ are $3_{1}$, $5_{1}$, $5_{2}$, $6_{3}$, a square knot, a granny knot, $8_{19}$ and $8_{20}$ (Calvo \cite{C01}). Therefore, Theorem \ref{main_theorem_recti} implies that at least one of them appears in the image of every rectilinear spatial embedding of $K_{3,3,1,1}$. On the other hand, it is known that the image of every rectilinear spatial embedding of $K_{7}$ contains a trefoil knot (Brown \cite{B77}, Ram{\'\i}rez Alfons{\'\i}n \cite{RA99}, Nikkuni \cite{N09b}). It is still open whether the image of every rectilinear spatial embedding of $K_{3,3,1,1}$ contains a trefoil knot. 
\end{Remark}


%
{\normalsize
}

\end{document}